\documentclass[twoside,leqno,10pt, A4]{amsart}
\DeclareMathAlphabet{\mathpzc}{OT1}{pzc}{m}{it}
\usepackage{amsfonts}
\usepackage{amsmath}
\usepackage{amscd}
\usepackage{amssymb}
\usepackage{amsthm}
\usepackage{amsrefs}
\usepackage{latexsym}
\usepackage{mathrsfs}
\usepackage{bbm}
\usepackage{enumerate}
\usepackage[mathscr]{eucal}
\usepackage{color}
\begin{document}

\newtheorem{theorem}[subsection]{Theorem}
\newtheorem{proposition}[subsection]{Proposition}
\newtheorem{lemma}[subsection]{Lemma}
\newtheorem{corollary}[subsection]{Corollary}
\newtheorem{conjecture}[subsection]{Conjecture}
\newtheorem{prop}[subsection]{Proposition}
\numberwithin{equation}{section}
\newcommand{\mr}{\ensuremath{\mathbb R}}
\newcommand{\mc}{\ensuremath{\mathbb C}}
\newcommand{\dif}{\mathrm{d}}
\newcommand{\intz}{\mathbb{Z}}
\newcommand{\ratq}{\mathbb{Q}}
\newcommand{\natn}{\mathbb{N}}
\newcommand{\comc}{\mathbb{C}}
\newcommand{\rear}{\mathbb{R}}
\newcommand{\prip}{\mathbb{P}}
\newcommand{\uph}{\mathbb{H}}
\newcommand{\fief}{\mathbb{F}}
\newcommand{\majorarc}{\mathfrak{M}}
\newcommand{\minorarc}{\mathfrak{m}}
\newcommand{\sings}{\mathfrak{S}}
\newcommand{\fA}{\ensuremath{\mathfrak A}}
\newcommand{\mn}{\ensuremath{\mathbb N}}
\newcommand{\mq}{\ensuremath{\mathbb Q}}
\newcommand{\half}{\tfrac{1}{2}}
\newcommand{\f}{f\times \chi}
\newcommand{\summ}{\mathop{{\sum}^{\star}}}
\newcommand{\chiq}{\chi \bmod q}
\newcommand{\chidb}{\chi \bmod db}
\newcommand{\chid}{\chi \bmod d}
\newcommand{\sym}{\text{sym}^2}
\newcommand{\hhalf}{\tfrac{1}{2}}
\newcommand{\sumstar}{\sideset{}{^*}\sum}
\newcommand{\sumprime}{\sideset{}{'}\sum}
\newcommand{\sumprimeprime}{\sideset{}{''}\sum}
\newcommand{\sumflat}{\sideset{}{^{\flat}}\sum}
\newcommand{\sumSTAR}{\sideset{}{^{\star}}\sum}
\newcommand{\shortmod}{\ensuremath{\negthickspace \negthickspace \negthickspace \pmod}}
\newcommand{\V}{V\left(\frac{nm}{q^2}\right)}
\newcommand{\sumi}{\mathop{{\sum}^{\dagger}}}
\newcommand{\mz}{\ensuremath{\mathbb Z}}
\newcommand{\leg}[2]{\left(\frac{#1}{#2}\right)}
\newcommand{\muK}{\mu_{\omega}}

\newcommand{\RR}{\mathbb{R}}
\newcommand{\QQ}{\mathbb{Q}}
\newcommand{\CC}{\mathbb{C}}
\newcommand{\NN}{\mathbb{N}}
\newcommand{\ZZ}{\mathbb{Z}}
\newcommand{\FF}{\mathbb{F}}
\newcommand{\C}{{\mathcal{C}}}
\newcommand{\OO}{{\mathcal{O}}}
\newcommand{\cc}{{\mathfrak{c}}}
\newcommand{\norm}{{\mathpzc{N}}}
\newcommand{\trace}{{\mathrm{Tr}}}
\newcommand{\ringO}{{\mathfrak{O}}}
\newcommand{\fa}{{\mathfrak{a}}}
\newcommand{\fb}{{\mathfrak{b}}}
\newcommand{\fc}{{\mathfrak{c}}}
\newcommand{\res}{{\mathrm{res}}}
\newcommand{\fp}{{\mathfrak{p}}}
\newcommand{\fm}{{\mathfrak{m}}}
\newcommand{\aut}{\rm Aut}
\newcommand{\mt}{m(t,u;\ell^k)}
\newcommand{\mtone}{m(t_1,u;\ell^k)}
\newcommand{\mttwo}{m(t_2,u;\ell^k)}
\newcommand{\mbadu}{m(t,u_0;\ell^k)}
\newcommand{\Sts}{S(t_1,t_2;\ell^k)}
\newcommand{\Stt}{S(t;\ell^k)}
\newcommand{\St}{R(t;\ell^k)}
\newcommand{\nN}{n(N,u;\ell^k)}
\newcommand{\Tnn}{T(N;\ell^k)}
\newcommand{\Tn}{T(N;\ell^k)}
\newcommand{\tilnul}{{\tilde{\nu}}_\ell}
\newcommand{\tilnulk}{{\tilde{\nu}}_\ell^{(k)}}
\makeatletter
\def\imod#1{\allowbreak\mkern7mu({\operator@font mod}\,\,#1)}
\makeatother

\title[Mean values of divisors of forms $n^2+Nm^2$]{Mean values of divisors of forms $n^2+Nm^2$}

\date{\today}
\author{Peng Gao and Liangyi Zhao}

\begin{abstract}
  Let $N$ be any fixed positive integer and define
\begin{align*}
   S_N(x)=\sum_{m, n \leq x}d(n^2+Nm^2),
\end{align*}
  where $d(n)$ is the divisor function.  We evaluate asymptotically $S_N(x)$ for several $N$, extending earlier work of Gafurov and Yu on
 the case $N=1$.
\end{abstract}

\maketitle



\section{Introduction}
\label{sec1}

  Let $d(n)$ denote the divisor function for any positive integer $n$. The study on the mean values of the divisor function has a long history. For example,
using what is now called the hyperbola method, Dirichlet established the following well-known asymptotic formula
\begin{equation} \label{diriprob}
  \sum_{n \leq x}d(n) = x\log x+(2\gamma_0-1)+O(x^{1/2}),
\end{equation}
where $\gamma_0$ is the Euler constant. Ever since Dirichlet's work, there have been
extensive efforts to improve the error term in \eqref{diriprob}. \newline

   We can regard $d(n)$ as evaluating the divisor function at the linear function $f(n)=n$. In this way, one may further consider the mean values of
$d(f(n))$ for an arbitrary polynomial $f(n)$. In \cite{Erdos}, P. Erd\"os established
the correct size of the main term for the mean values of $d(f(n))$ for any irreducible $f(x) \in \mz[x]$ by showing that
\begin{equation} \label{asymdiv}
 \sum_{n \leq x}d(f(n)) \asymp x\log x
\end{equation}
with the implied constants depending on $f$. \newline

For quadratic polynomials, one can obtain asymptotic formulas for the left-hand side of \eqref{asymdiv}.  Such formulas were given by
E. J. Scourfield \cite{Scot} and a sharper result was obtained by C. Hooley in \cite{Hooley}, who showed that when $-c$ is not a perfect square,
\begin{align*}
   \sum_{n \leq x}d(n^2+c) = \lambda_1 x\log x+\lambda_2x+(x^{8/9}(\log x)^3).
\end{align*}
   A more compact representation of $\lambda_1$ is given by J. Mckee in \cite{McKee}. \newline

   In view of the above results on sums of divisor functions over quadratic polynomials, it is natural to extend the study to sums of divisor functions over
quadratic forms. For binary quadratic forms, this is first studied by N. Gafurov, who obtained asymptotic formulas for
\begin{align*}
   S_1(x)=\sum_{m, n \leq x}d(n^2+m^2).
\end{align*}
   A more accurate formula for the above sum was later given by G. Yu in \cite{Yu}. \newline

   In \cite{C&V}, C. Calder\'on and M. J. de Velasco obtained an asymptotic formula for the mean value of divisor functions over certain ternary quadratic form.
The error term was improved by  R.T. Guo and W.G. Zhai in \cite{G&Z} using the circle method. Also using the circle method, L. Q. Hu
obtained an asymptotic formula for the mean value of divisor functions over certain quaternary quadratic form in \cite{Hu}. \newline

For any positive integer $N$ and real number $x$, we define
\begin{align} \label{SN}
  S_N(x)=\sum_{m, n \leq x}d(n^2+Nm^2).
\end{align}
We shall give asymptotic formulas for $S_N(x)$ for certain values of $N$.  The result in this paper is motivated by the above mentioned work of Gafurov and Yu who dealt with the case $N=1$.  Our result is
\begin{theorem}
\label{mainthm}
  For $N=2, 67$ or $163$ and $x \geq 1$, we have
\begin{align*}
  S_N(x)=A_1(N)x^2\log x+A_2(N)x^2+O(x^{3/2+\varepsilon}),
\end{align*}
  where $A_1(N), A_2(N)$ are constants given in \eqref{A1N} and \eqref{A2N}, respectively.
\end{theorem}

   Our proof of Theorem \ref{mainthm} follows the treatment of Theorem 1 in \cite{Yu}. Note that the set of values of $N$ given in the statement of
Theorem \ref{mainthm} forms a
subset of the following set:
\begin{align}
\label{N}
  \{ 1,2,3,7,11,19,43,67,163 \}.
\end{align}
 For each $N$ in the above set,
the corresponding imaginary quadratic number fields $K=\mq(\sqrt{-N})$ has class number $1$. This allows us to establish a bijection
between the roots of $v^2+N \equiv 0 \pmod d$ and the representations of $d$ for $(d, 2N)=1$ in terms of
norms of elements in the ring of integers of $\mq(\sqrt{-N})$ (see Lemma \ref{repncorresp} below). We are then able to establish a large sieve result
to estimate certain exponential sums involving the quadratic roots (see Lemma \ref{approx} below).
This in turn leads to the desired error term in Theorem \ref{mainthm}. The reason we cannot take $N$ from the full set in \eqref{N} in Theorem \ref{mainthm} is a certain restriction on the size of $N$ in the proof of Lemma \ref{approx}.  In fact, as one can see from the proof of Theorem \ref{mainthm}, the assertion of Theorem \ref{mainthm} remains valid for any $N$ in the set \eqref{N} so long as Lemma \ref{approx} can be established for it.

\subsection{Notations} The following notations and conventions are used throughout the paper.\\
\noindent
\noindent $e(z) = \exp (2 \pi i z) = e^{2 \pi i z}$. \newline
$[x]$ stands for the largest integer not exceeding $x$ and let $\{x\}=x-[x]$. \newline
We define $\psi(x)= \{ x \}-1/2$ and $\| x \|=\min \{ \{x \}, 1-\{ x \} \}$. \newline
\noindent
$f =O(g)$ or $f \ll g$ means $|f| \leq cg$ for some unspecified positive constant $c$. \newline
$n \sim N$ means there are positive constants $c_1, c_2$ such that $c_1N < n \leq c_2N$.

\section{Preliminary Lemmas} \label{sec 2}

Our first three lemmas below aim to establish certain large sieve result involving roots of quadratic congruences $v^2+N \equiv 0 \pmod d$. 
We let $K=\mq(\sqrt{-N})$ and $\mathcal{O}_K$  be the ring of integers in $K$. As a preparation, our first lemma
 characterizes these roots in terms of representations of $d$ as norms of elements in $\mathcal{O}_K$  .
\begin{lemma}
\label{repncorresp}
   Let $N$ be a fixed integer given in \eqref{N} and $d$ any positive integer. For $(d, 2N)=1$, there is a one-to-one correspondence between the roots $v$ of
\begin{align}
\label{congeqn}
   v^2+N \equiv 0 \pmod d
\end{align}
 and the representations
\begin{align}
\label{drepn}
  d=\displaystyle \begin{cases}
    \displaystyle  r^2+Ns^2, \quad (r,s)=1, \; r>0 & \text{if $N=1, 2$}, \\ \\
    \displaystyle \frac {r^2+Ns^2}{4}, \quad (r,s) \leq 2, \; r>0, \quad  r\equiv s \pmod 2 & \text{otherwise}, \\
    \end{cases}
\end{align}
  given by
\begin{align}
\label{vsoln}
   v \equiv \pm r\overline{s} \pmod d,
\end{align}
   where $s\overline{s} \equiv 1 \pmod d$. \newline

   When $N \neq 2$ and $2|d$, $(d, N)=1$, then \eqref{congeqn} is not solvable for $N \neq 7$ when $2^3 | d$. If $N=7$ and $2^3|d$,
then there is a one-to-one correspondence between the roots $v$ of \eqref{congeqn} modulo $d$ and
the representations given in \eqref{drepn} for $d$ with $(r,s)=2$ and $2d$ with $(r,s)=1$. The correspondence is given by \eqref{vsoln} modulo $d$,
except that when $(r,s)=2$, we replace $r,s$ by $r/2, s/2$ in \eqref{vsoln}. \newline

   When $N|d$, then \eqref{congeqn} is not solvable if $N^2|d$. If $d=Nd_1$ with $(d_1, N)=1$, then there is a one-to-one correspondence between the roots $v$ of \eqref{congeqn} and the roots $v_1$ of \eqref{congeqn} with $d$ replaced by $d_1$ there. The correspondence is given by $v \equiv N\overline{v}_1 \pmod {d}$.
\end{lemma}
\begin{proof}
   We note first that the cases with $N=1, 3$ of the lemma are established in \cite{Fo&I} and \cite{Pandey}, respectively. Moreover, the treatment of $N=2$ is similar but easier compared to the remaining cases. We may thus assume that $N \neq 1, 2, 3$ in what follows.  The remaining $N$'s in \eqref{N} all satisfy $-N \equiv 1 \pmod 4$ so that any algebraic integer in $\mathcal{O}_K$ has the form
\begin{align*}
   \frac {r+s\sqrt{-N}}{2}, \quad r, s \in \mz, \quad r \equiv s \pmod 2.
\end{align*}

    When $d$ is odd, we write $d=\prod^n_{i=1} p^{\alpha_i}_i$, with $p_i$ being distinct primes and $\alpha_i$ being positive integers. By the Chinese Remainder Theorem, \eqref{congeqn} is solvable if and only if the equations
\begin{align} \label{congeqnmodpp}
   v^2+N \equiv 0 \pmod {p^{\alpha_i}_i}
\end{align}
   are solvable for all $i$ and the number of solutions of \eqref{congeqn} equals the product of the number of solutions of the above equations. Further, by Hensel's lemma \cite[Theorem 3.19]{Na}, \eqref{congeqnmodpp} are solvable if and only if
\begin{align}
\label{congeqnmodp}
   v^2+N \equiv 0 \pmod {p_i}
\end{align}
   are solvable for all $i$.  If solvable, the number of solutions for each equation given in \eqref{congeqnmodp} is two so that
each equation in \eqref{congeqnmodpp} admits two solutions as well, again by Hensel's lemma. \newline

Now for each $i$, equation \eqref{congeqnmodp} is solvable if and only if
\begin{align}
\label{quad}
  \leg {N}{p_i}=1,
\end{align}
   where $\leg {\cdot}{p_i}$ is the Jacobi symbol modulo $p_i$.  It is well-known (see, for example \cite[Section 3.8]{iwakow}) that \eqref{quad} holds if and only if
$(p_i)$ splits in $K$, i.e.
\begin{align}
\label{pdecomp}
     (p_i)=\mathfrak{p}_i\overline{\mathfrak{p}}_i, \quad  \mathfrak{p}_i \neq \overline{\mathfrak{p}}_i,
\end{align}
   where $\mathfrak{p}_i, \overline{\mathfrak{p}}_i$ are prime ideals of $K$. Note that the above relation implies that $\mathcal{N}(\mathfrak{p}_i)=p_i$,
where $\mathcal{N}$ denotes the norm for $K$ over $\ratq$. As $K$ is of class number one, $\mathfrak{p}_i$ is a principal ideal. Choosing
a generator $(r_i+s_i\sqrt{-N})/2$ with $r_i>0$ (note that $r_i \neq 0$ here for otherwise this implies that $N | \mathcal{N}(\mathfrak{p}_i)=p_i$, contradicting the assumption that $(d, N)=1$) for $\mathfrak{p}_i$ implies that $p_i=(r^2_i+Ns^2_i)/4$. Moreover, we have $(r_i, s_i) \leq 2$ since if we have $2^k, k \geq 2$ orany rational prime $q \neq 2$ dividing $(r_i, s_i)$, then taking the norms implies that $2^{2k-2} | p_i$ or $q^2|p_i$, which is not possible. \newline

   Conversely, if one can writes $p_i=(r^2_i+Ns^2_i)/4$ with $(r_i, s_i) \leq 2, r_i \equiv s_i \pmod 2$, then \eqref{congeqn} is solvable
via the solutions given by \eqref{vsoln}. Note that we have
$$(p_i)=\mathfrak{p}_i\overline{\mathfrak{p}}_i=\left(\frac {r_i+s_i\sqrt{-N}}{2} \right )\left(\frac {r_i-s_i\sqrt{-N}}{2}\right )$$
 in $K$. This implies that
$(\frac {r_i+s_i\sqrt{-N}}{2})=\mathfrak{p}_i$ or $\overline{\mathfrak{p}_i}$ by unique factorization in $K$. It follows that
the pair $(r_i, s_i)$ is uniquely determined up to units (note that in our case the only units in $K$ are $\pm 1$).
We then deduce that in this case the correspondence \eqref{vsoln} is indeed one-to-one. Thus, the assertion of the lemma is valid for $d$ being a prime. \newline

   When $d$ is a prime power, say $d=p^{\alpha_i}_i$,
then we see that when \eqref{congeqnmodpp} is solvable, \eqref{congeqnmodp} is also solvable so that \eqref{pdecomp} is valid
 and similar to our discussions above,
a generator $(u_i+v_i\sqrt{-N})/2$ for $\mathfrak{p}^{\alpha_i}_i$ satisfies $(u_i, v_i) \leq 2, u_i \equiv v_i \pmod 2$
and that $(u^2_i+Nv^2_i)/4=p^{\alpha_i}_i$. Conversely, if we have
$(u^2_i+Nv^2_i)/4=p^{\alpha_i}_i$ for some $(u_i, v_i) \leq 2, u_i \equiv v_i \pmod 2$. Then \eqref{congeqn} is solvable and we have
$$(\mathfrak{p}_i\overline{\mathfrak{p}}_i)^{\alpha_i}=\left(\frac {u_i+v_i\sqrt{-N}}{2} \right)\left(\frac {u_i-v_i\sqrt{-N}}{2} \right),$$ which
implies that we must have $\left( \frac {u_i+v_i\sqrt{-N}}{2} \right)=\mathfrak{p}_i^{\alpha_i}$ or $\overline{\mathfrak{p}}_i^{\alpha_i}$.
This in turn implies that the pair $(u_i, v_i)$ is uniquely determined up to units so that
in this case the correspondence \eqref{vsoln} is also one-to-one. Thus, the assertion of the lemma is valid for $d$ being a prime power. \newline

   Now, to prove the assertion of the lemma for a general $d$, we first show that when \eqref{drepn} is valid, then
the solutions given in \eqref{vsoln} are all different modulo $d$. Suppose now we have
\begin{align}
\label{didentity}
   d=\frac {r^2+Ns^2}{4}=\frac {(r')^2+N(s')^2}{4}, \quad (r, s) \leq 2, \; (r', s') \leq 2, \; r\equiv s (\bmod 2), \; r' \equiv s' (\bmod 2), r, \; r'>0
\end{align}
   and that
\begin{align}
\label{congcond}
   r\overline{s} \equiv r'\overline{s'} \pmod d.
\end{align}
   This implies that
\begin{align*}
   rs' \equiv r's \pmod d.
\end{align*}

   As the above congruence relation also holds with $d$ being replaced by $2$ and we have $(d,2)=1$,
    we deduce that $2d| (rs'-r's)$. We further note (as arithmetic means always exceeds geometric means) that
\begin{align*}
   |rs'-r's| \leq |rs'|+|r's| \leq \frac {r^2+N(s')^2}{2\sqrt{N}}+\frac {(r')^2+Ns^2}{2\sqrt{N}}=\frac {4d}{\sqrt{N}} < 2d
\end{align*}
   as $N>4$. We then conclude that we must have $rs'=r's$. Thus we have $s | rs'$ and $s' | r's$. As $(r,s)=(r',s')=1$, we deduce that $s | s'$ and $s'|s$.  Hence $s=s'$ (we can not have $s =-s'$ as this would imply that $r=-r'$ but both $r$ and $r'$ are positive)
and then $r=r'$ as they are both positive. If $(r,s)=1$ and $(r',s')=2$, then on replacing $r', s'$ by $r'/2, s'/2$ and arguing as above,
we deduce that $r'=2r, s'=2s$, contradicting \eqref{didentity}. We can discuss other cases similarly to conclude that
in order for the congruence condition \eqref{congcond} to hold, we must have $r=r', s=s'$. Hence the solutions given in \eqref{vsoln} are all distinct. \newline

Now we return to the case of a general $d$. When \eqref{congeqn} is solvable, then \eqref{congeqnmodpp} are solvable for all $p_i$ so that we have \eqref{pdecomp} and
similar to our discussions above, one checks that all the
pairs $(r,s)$ with $r>0, (r,s) \leq 2$ satisfying \eqref{drepn} are coming from identifying $(r+s\sqrt{-N})/2$ with a generator of
$\prod^n_{i=1}\varpi^{\alpha_i}_i$ with $\varpi_i=\mathfrak{p}_i$ or $\overline{\mathfrak{p}}_i$. We can now fix a generator $(r_i+s_i\sqrt{-N})/2$ with
$r_i>0$ for each $\mathfrak{p}_i^{\alpha_i}$.
By unique factorization, we must have $(r+s\sqrt{-N})/2=\prod^n_{i=1}d_i$ with $d_i=(r_i+s_i\sqrt{-N})/2$ or $(r_i-s_i\sqrt{-N})/2$. There are $2^n$ ways of forming such a product and as conjugated pairs determine the value of $s$ up to a sign, the total number of ways to obtain different pairs of $(r,s)$ up to the sign of $s$ are $2^n/2$. As we have
shown above, different pairs give different pairs of solutions to \eqref{congeqn} via \eqref{vsoln}.  Thus, we obtain $2^n$ different solutions to \eqref{congeqn} via this correspondence. On the other hand, by the Chinese Remainder Theorem, we know \eqref{congeqn} has exactly $2^n$ solutions. Thus, the solutions given in \eqref{vsoln} are in one-to-one correspondence to the solutions of \eqref{congeqn} and this completes the proof for the general case. \newline

   Next, we examine the case when $N \neq 2$ and $2|d$. Note first that any $N \neq 7$ given in \eqref{N} is $\equiv 3 \pmod 8$
so that $v^2+N \equiv 0 \pmod 8$ has no solutions. Thus equation \eqref{congeqn} is not solvable when $2^3|d$.   When $N=7$,
equation \eqref{congeqn} is always solvable when $d=2^k$. In fact, one checks that it has one solution when $k=1$,
two solutions when $k=2$ and four solutions when $k \geq 3$ using arguments similar to those in the proof \cite[Theorem 3.19]{Na} of Hensel's lemma. As
any representation of $d$ implies that \eqref{congeqn} is solvable via \eqref{vsoln}, we may now assume that \eqref{congeqn} is solvable modulo
$d=2^k\prod^n_{i=1} p^{\alpha_i}_i$
with $k \geq 3$, $p_i$ being distinct odd primes and $\alpha_i$ being positive integers. In this case we still have \eqref{pdecomp} and we further note that we have
\begin{align*}
  (2)=\mathfrak{p}\overline{\mathfrak{p}}, \quad \mathfrak{p}= \left( \frac {1+\sqrt{-7}}{2} \right), \quad \mathfrak{p} \neq \overline{\mathfrak{p}}.
\end{align*}
   Now, all representation of $d$ satisfying \eqref{drepn} are coming from identifying $(r+s\sqrt{-7})/2$ as a generator of an ideal such that
\begin{align}
\label{mod2gen}
   \left ( \frac {r+s\sqrt{-7}}{2} \right )=\varpi\prod^n_{i=1}\varpi^{\alpha_i}_i
\end{align}
  with $\varpi_i=\mathfrak{p}_i$ or $\overline{\mathfrak{p}}_i$,
$\varpi=\mathfrak{p}^k, \overline{\mathfrak{p}}^k, 2\mathfrak{p}^{k-2}$ or $2\overline{\mathfrak{p}}^{k-2}$.
In fact, it is easy to see that no other product will produce a representation in \eqref{drepn} satisfying $(r,s) \leq 2$.
On the other hand, let $r,s$ be given in \eqref{mod2gen}, if $2^2 |(r,s)$, then by setting $r=4r_1, s=4s_1$ and taking norms on both sides of \eqref{mod2gen},
we see that $2^{k-2} |r^2_1+7s^2_1$. As $k \geq 3$, this implies that $r_1 \equiv s_1 \pmod 2$. Thus,
$(r_1+s_1\sqrt{-7})=2((r_1+s_1\sqrt{-7})/2)$ with $(r_1+s_1\sqrt{-7})/2 \in \mathcal{O}_K$,
contradicting \eqref{mod2gen}. We conclude that every $r,s$ given by \eqref{mod2gen} satisfies $(r,s) \leq 2$. From this we also see that, if
$\varpi=2\mathfrak{p}^{k-2}$ or $2\overline{\mathfrak{p}}^{k-2}$, then $2|(r, s)$ for every $r,s$ given by \eqref{mod2gen},
so that in this case $(r,s)=2$. On the other hand, if $\varpi=\mathfrak{p}^k$ or $\overline{\mathfrak{p}}^k$ then $(r,s)=1$ for every $r,s$
given by \eqref{mod2gen}. For if $2 |(r,s)$, then by setting $r=2r_1, s=2s_1$ and taking norms on both sides of \eqref{mod2gen},
we see that $2^{k} |r^2_1+7s^2_1$. As $k \geq 3$, this implies that $r_1 \equiv s_1 \pmod 2$. Thus,
$(r_1+s_1\sqrt{-7})=2((r_1+s_1\sqrt{-7})/2)$ with $(r_1+s_1\sqrt{-7})/2 \in \mathcal{O}_K$, again contradicting \eqref{mod2gen}. \newline

  We then conclude that for $d=2^k\prod^n_{i=1} p^{\alpha_i}_i$ with $k \geq 3$, there are exactly $2^{n+1}$ ways to represent $d$ as in \eqref{drepn} with $(r,s)=1$
and also exactly $2^{n+1}$ ways to represent $d$ as in \ref{drepn} with $(r,s)=2$. We now take corresponding to the representation of $d$
the $2^{n+1}$ ways to obtain different pairs of $(r,s)$ (up to the sign of $s$) with $(r,s)=1$ and corresponding to the representation of $2d$
the $2^{n+1}$ ways to obtain different pairs of $(r,s)$ (up to the sign of $s$) with $(r,s)=1$. This way we have a total of $2^{n+2}$ ways of obtaining solutions $v$ via \eqref{vsoln}.  Similar to our arguments above, one then shows that these roots are all distinct. 
As $2^{n+2}$ equals the total number of solutions of \eqref{congeqn}, the one-to-one
correspondence thus follows in this case. \newline

   Lastly, when $N |d$, we note that equation \eqref{congeqn} has only one solution $v \equiv 0 \pmod d$ when 
$d=N$ and no solutions when $d=N^n$ with $n \geq 2$.
It follows that in this case \eqref{congeqn} is solvable only when $d=Nd_1$ with $(N, d_1)=1$. As any solution $v$ must satisfy $v \equiv 0 \pmod N$ and 
$(v/N, d_1)=1$, we can write
$v=N\overline{v}_1$ to derive the desired conclusion.
\end{proof}

Now, using Lemma \ref{repncorresp}, we can establish the following
\begin{lemma}
\label{approx}
   Let $N \in \{ 1, 2, 67, 163 \}$. For each solution $v$ of \eqref{congeqn}, there exists a fraction $a_v/q_v$
with $(q_v, a_v)=1, q_v>0$ such that $c_1 d^{1/2} \leq q_v \leq c_2d^{1/2}$ with the constants $c_1, c_2$ the choice of whose values depend on $N$ only, and
\begin{align}
\label{wellapprox}
    \Big | \frac {v}{d} -\frac {a_v}{q_v} \Big | \leq \frac 1{q^2_v}.
\end{align}
\end{lemma}
\begin{proof}
   As the case $N=1,2$ is easier, we may assume that $N=67$ or $163$. Note that $N$ belongs to the set given in \eqref{N}. When $N|d$, then Lemma \ref{repncorresp} implies that every
solution $v$ of equation \eqref{congeqn} can be written as $v \equiv N\overline{v}_1 \pmod {Nd_1}$ with $(N, d_1)=1$. In this case it suffices to establish
\eqref{wellapprox} with $v/d$ replaced by $v_1/d_1$. Hence, it remains to prove \eqref{wellapprox} for both $v$ and $\overline{v}$ for $(d,N)=1$. As the two cases
are similar, we shall henceforth prove \eqref{wellapprox} for $(d,N)=1$. \newline

   Suppose that $d=2^l d'$ with $l=1$ or $2$ and $(d',2)=1$. Then it is easy to see that the solutions to \eqref{congeqn} can be written as
$v'+d'j$ for some $0 \leq j \leq 2^{l}-1$ with $v'$ being the solution to equation \eqref{congeqn} modulo $d'$. Note that we have
\begin{align*}
  \Big | \frac {v'+d'j}{2^l d'} -\frac {a_{v'}+jq_{v'}}{2^l q_{v'}}  \Big |=\frac 1{2^l}
\Big | \frac {v'}{d'} -\frac {a_{v'}}{q_{v'}}\Big |.
\end{align*}
   By writing $a_v/q_v= (a_{v'}+jq_{v'})/(2^l q_{v'})$ with $(a_v, q_v)=1$, we see that
in this case the assertion of the lemma holds for $v$ with some $q_v | 2^l q_{v'}$ provided that we can show that the assertion of the lemma holds for all $(d, 2N)=1$
with \eqref{wellapprox} being replaced by
\begin{align*}
   \frac 1{2^l} \Big | \frac {v}{d} -\frac {a_{v}}{q_{v}}\Big | \leq \frac 1{(2^lq_v)^2}.
\end{align*}

   As the case $l=2$ of the above inequality implies the case $l=0, 1$ (the case $l=0$ being inequality \eqref{wellapprox}). It suffices to establish the
above inequality for $l=2$.
Note that for each $d, (d, 2N)=1$ such that $v^2 +N \equiv 0 \pmod d$ is solvable, there is by Lemma \ref{repncorresp}
a one-to-one correspondence between solutions $v  \pmod d$ to $v^2 +N \equiv 0 \pmod d$
and representations of $d$ given by \eqref{drepn}. If we replace $s, r$ in \eqref{drepn} by $s_v$ and $r_v$
respectively to indicate the dependence on the corresponding $v$, then for any given $r_v, s_v$, we can take
for the corresponding $v \pmod d$ to be the residue class $\pm r_v \overline{s}_v \pmod d$. When $(r_v, s_v)=2$, we may further replace $r_v$ and $s_v$ by $r_v/2$, $s_v/2$ (note that $d$ is odd). Thus, we may assume in what follows that $(r_vs_v, 2)=1$ satisfying $r^2_v+Ns^2_v=4d$.
 For any solution $v$ of \eqref{congeqn} given by \eqref{vsoln}, it follows from \eqref{drepn} that
\begin{align*}
  r^2_v+Ns^2_v \equiv 0 \pmod d.
\end{align*}
  This implies that
\begin{align}
\label{rs}
  r_v\overline{s}_v \equiv -N \overline{r}_v s_v  \pmod d.
\end{align}

   It follows from this and  \eqref{vsoln} that
\begin{align*}
  \frac {v}{d} \equiv \frac {r_v\overline{s}_v}{d} \equiv \frac {r_vs_v\overline{s}_v}{ds_v} \pmod 1.
\end{align*}
   We write $s_v\overline{s}_v=1+ad$ with $a \in \mz$ to get
\begin{align*}
  \frac {v}{d} \equiv \frac {ar_v}{s_v}+\frac {r_v}{ds_v} \pmod 1.
\end{align*}

  Note that
\begin{align*}
  -4 \equiv 4ad \equiv 4a(r^2_v+Ns^2_v) \equiv 4ar^2_v \pmod {s_v}.
\end{align*}
   As $(r_vs_v,2)=1$, it follows that
\begin{align*}
  ar_v \equiv -\tilde{r_v} \pmod {s_v},
\end{align*}
   where $\tilde{r_v}$ denotes the number satisfying $0< \tilde{r_v}<|s_v|$ and $r_v\tilde{r_v} \equiv 1 \pmod {s_v}$. \newline

   We then deduce that
\begin{align}
\label{sdenom}
  \frac {v}{d} \equiv -\frac {\tilde{r_v}}{s_v}+\frac {r_v}{ds_v} \pmod 1.
\end{align}

    Similar, we deduce from \eqref{rs} that
\begin{align*}
  \frac {v}{d} \equiv \frac {-Ns_v\overline{r}_v}{d} \equiv \frac {-Ns_vr_v\overline{r}_v}{dr_v} \pmod 1.
\end{align*}
   We write $r_v\overline{r}_v=1+a'd$ with $a' \in \mz$ to get
\begin{align*}
  \frac {v}{d} \equiv \frac {-a'Ns_v}{r_v}-\frac {Ns_v}{dr_v} \pmod 1.
\end{align*}

  Note that, as before,
\begin{align*}
  -4\equiv 4a'd \equiv 4a'(r^2_v+Ns^2_v) \equiv 4a'Ns^2_v \pmod {r_v}.
\end{align*}
   As $(r_vs_v, 2)=1$, it follows that
\begin{align*}
   -a'Ns_v \equiv \tilde{s_v} \pmod {r_v},
\end{align*}
where $\tilde{s_v}$ denotes the number satisfying $0< \tilde{s_v}<|r_v|$ and $s_v\tilde{s_v} \equiv 1 \pmod {r_v}$. \newline

   We then deduce that
\begin{align} \label{rdenom}
  \frac {v}{d} \equiv \frac {\tilde{s_v}}{r_v}-\frac {Ns_v}{dr_v} \pmod 1.
\end{align}

   Combining \eqref{sdenom} and \eqref{rdenom}, we see that there exists integers  $\alpha(v, d)$, $\beta(v, d)$ such that
\begin{align}
\label{vdapprx}
  \Big | \frac {v}{d} +\frac {\tilde{r}_v+\alpha(v,d)s_v}{s_v}\Big |=\frac {|r_v|}{d|s_v|} \leq \frac 1{4s^2_v}, \quad
  \Big | \frac {v}{d} -\frac {\tilde{s}_v+\beta(v,d)r_v}{r_v}\Big |=\frac {N|s_v|}{d|r_v|} \leq \frac 1{4r^2_v},  
\end{align}
where the first inequality above follows by noting that
\begin{align*}
    |r_vs_v| \leq \frac {r^2_v+Ns^2_v}{2\sqrt{N}} = \frac {2d}{\sqrt{N}} \; \mbox{and} \; N \geq 67 > 8^2.
\end{align*}
This marks the restriction on the size of $N$ the precludes us from taking $N$ from the full set in \eqref{N}.  The second inequality in \eqref{vdapprx} is easily seen to be valid when $|s_v| \leq \gamma(N) |r_v|$ for some fixed sufficiently small positive constant $\gamma(N)$ depending only on $N$.
We can now take $a_v/q_v$ with $(a_v, q_v)=1$ such that $a_v/q_v=-(\tilde{r}_v+\alpha(v,d)s_v)/s_v$ when  $|s_v| \geq \gamma(N) |r_v|$ and $a_v/q_v=(\tilde{s}_v+\beta(v,d)r_v)/r_v$
when $|s_v| \leq \gamma(N) |r_v|$. This completes the proof of the lemma. 
\end{proof}

  We now use Lemma \ref{approx} to establish the following large sieve result, which is a generalization of Lemma 2 in \cite{Yu}:
\begin{lemma} \label{SDHN}
   Let $N \in \{ 1,2, 67, 163 \}$.  For positive integers $D, H, M$, set
\begin{align*}
   S(D,H,M;N)=\sum_{d \sim D}\sum_{\substack{v \bmod d \\ v^2+N\equiv 0 \bmod d}}\sum_{h \leq H}\frac 1h \left| \sum_{n \leq M} e \left( \frac {hnv}{d} \right) \right|.
\end{align*}
   Then for $D$ sufficiently large, $H,M > 3$ and any $\varepsilon>0$, we have
\begin{align*}
   S(D,H,M;N) \ll_{\varepsilon}(DMH)^{\varepsilon}(D+M)\sqrt{D}.
\end{align*}
\end{lemma}
\begin{proof}
Using the well-known bound for geometric sums, 
\begin{align*}
  \sum_{n \leq M} e \left( \frac {hnv}{d} \right) \ll \min \left( M, \frac 1{\|hv/d \|} \right),
\end{align*}
we deduce that
\begin{align} \label{Sbound}
   S(D,H,M; N) \ll \sum_{d \sim D}\sum_{\substack{v \bmod d \\ v^2+N\equiv 0 \bmod d}}\sum_{h \leq H}\frac 1h\min \left( M, \frac 1{\|hv/d \|} \right).
\end{align}

   We note that \cite[Lemma 1]{Yu} implies that for any real number $\alpha$ satisfying
\begin{align*}
  \Big | \alpha -\frac a{q} \Big | \leq \frac 1{q^2},
\end{align*}
  with $(a, q)=1, q>0$, then
\begin{align*}
   \sum_{t \leq T} \min \left( \frac M{t}, \frac 1{\| \alpha t \| } \right) \ll \left( Mq^{-1}+T+q \right)\log (2qT).
\end{align*}

  We now apply the above bound with $\alpha=v/d$ to the right-hand side of \eqref{Sbound} and Lemma \ref{approx} to get that (breaking $h$ into dyadic intervals and mindful of the relation $q_v \sim d^{-1/2}$)
\begin{align*}
   S(D,H,M;N) \ll & \log H \max_{J \ll H}J^{-1}\sum_{d \sim D}\sum_{\substack{v \bmod d \\ v^2+N\equiv 0 \bmod d}}
\sum_{h \sim J}\min \Big(\frac {MJ}{h}, \frac {1}{\|hv/d \|} \Big) \\
\ll & (DMH)^{\varepsilon}\max_{J \ll H}J^{-1} \sum_{d \sim D}\sum_{\substack{v \bmod d \\ v^2+N\equiv 0 \bmod d}} \left( \frac{MJ}{|q_v|}+J+|q_v| \right) \\
\ll & (DMH)^{\varepsilon}\sum_{d \sim D}\sum_{\substack{v \bmod d \\ v^2+N\equiv 0 \bmod d}}(Md^{-1/2}+d^{1/2}) \ll (DMH)^{\varepsilon}(D+M)\sqrt{D}.
\end{align*}
  This completes the proof of the lemma. 
\end{proof}

   We define $\rho_N(d)$ to be the number of solutions of the congruence $u^2+Nv^2 \equiv 0 \pmod d$ subject to $0<u, v \leq d$. Let $F_N(s)$ be the
Dirichlet series associated with $\rho_N(n)/n$:
\begin{align*}
  F_N(s)=\sum^{\infty}_{n=1}\frac {\rho_N(n)/n}{n^s}, \quad \Re(s) \geq 2.
\end{align*}

    Note that $\rho_N(n)$ is multiplicative.  Thus it suffices to determine the values of $\rho_N$ at prime powers.  For any prime $p$ and integer $\alpha \geq 1$, suppose thatt $(v, p^{\alpha})=p^{\beta}$ with $\beta \leq [\alpha/2]$ so that
\begin{align}
\label{uveqn}
  u^2+Nv^2 \equiv 0 \pmod {p^{\alpha}}, \quad 0< u, v \leq p^{\alpha},
\end{align}
  is equivalent to
\begin{align*}
  p^{2\beta}(u^2+Nv^2) \equiv 0 \pmod {p^{\alpha}}, \quad 0< u, v \leq p^{\alpha-\beta}, (uv, p)=1.
\end{align*}
   The above congruence is further equivalent to
\begin{align*}
   u^2+Nv^2 \equiv 0 \pmod {p^{\alpha-2\beta}}, \quad 0< u, v \leq p^{\alpha-\beta}, (uv, p)=1.
\end{align*}
   It is easy to see that the above equation has $\varphi(p^{\alpha-2\beta})\rho_{0, N}(p)$ solutions satisfying $0< u, \; v \leq p^{\alpha-2\beta}, (uv, p)=1$, where for any integer $d$, 
\begin{equation} \label{rho0def}
\rho_{0,N}(d) = \# \left\{ x \bmod{d} : x^2+N \equiv 0 \pmod d \right\} . 
\end{equation}
It follows that the total number of solutions of the above equation is $p^{2\beta}\varphi(p^{\alpha-2\beta})\rho_{0,N}(p)=\varphi(p^{\alpha})\rho_{0, N}(p)$. \newline

   On the other hand, if $p^{\beta}|(v, p^{\alpha})$ with $\beta = [\alpha/2]+1$, equation \eqref{uveqn} is equivalent to
\begin{align*}
  p^{2\beta}(u^2+Nv^2) \equiv 0 \pmod {p^{\alpha}}, \quad 0< u, \; v \leq p^{\alpha-\beta}.
\end{align*}
  The above equation has $p^{2(\alpha-\beta)}$ solutions. \newline

  Summarizing our discussions above, we get
\begin{align*}
  \rho_N(p^{\alpha})=
\displaystyle \begin{cases}
    \displaystyle  \sum^{\alpha/2-1}_{\beta=0}\varphi(p^{\alpha})\rho_{0,N}(p)+p^{\alpha},  & \mbox{if} \; 2|\alpha, \\ \\
    \displaystyle \sum^{(\alpha-1)/2}_{\beta=0}\varphi(p^{\alpha})\rho_{0,N}(p)+p^{\alpha-1}, & \mbox{if} \; 2 \nmid \alpha. \\
    \end{cases}
\end{align*}

    One checks that for $(p, 2N)=1$,  $\rho_{0,N}(p)=1+\chi_N(p)$, where $\chi_N=\leg {-N}{\cdot}$ is the Kronecker symbol.
 Combining this with the above expression for $\rho_N(p^{\alpha})$, we deduce that for $\Re(s) \geq 2$,
\begin{align*}
  F_N(s)=\zeta(s)L(s, \chi_N)G_N(s),
\end{align*}
   where $\zeta(s)$ is the Riemann zeta function, $L(s, \chi_N)$ is the Dirichlet $L$-function associated with $\chi_N$ and $G_N(s)$ is given by the Euler product
\begin{align*}
  G_N(s)=\prod_p G_{p,N}(s) \quad  \mbox{with} \quad  G_{p,N}(s)=1-\frac {\chi_N(p)}{p^{s+1}},  \quad (p, 2N)=1.
\end{align*}

   We now proceed similarly as in the proofs of \cite[Lemmas 4, 6, 7]{Gaf2} to obtain the following
\begin{lemma}
\label{rhod}
   Suppose $y \geq 3$, we have
\begin{align*}
   \sum_{d \leq y}\rho_N(d) =& A_Ny^2+O\left( y^{4/3}(\log y)^2 \right), \quad
   \sum_{d \leq y}\frac {\rho_N(d)}{d} = 2A_Ny+O \left( y^{1/3}(\log y)^2 \right), \\
    \sum_{d \leq y}\frac {\rho_N(d)}{d^2} =& 2A_N\log y+2\int\limits^{\infty}_1\frac {E_N(t)}{t^2}dt+A_N+O\left( y^{-2/3}(\log y)^2 \right),
\end{align*}
   where
\begin{align}
\label{EN}
   E_N(t)=\sum_{d \leq t}\rho_N(d)-A_Nt^2 \ll t^{4/3}(\log t)^2,
\end{align}
   and
\begin{align}
\label{A}
   A_N=\frac {L(1, \chi_N)G_N(2)}{2}.
\end{align}
\end{lemma}

\section{Proof of Theorem \ref{mainthm}}
\label{sec 3}

    Recalling the definition of $S_N(x)$ from \eqref{SN}, we have
\begin{align*}
  S_N(x)=\sum_{\substack{m, n \leq x\\ n^2+Nm^2=kl}}1.
\end{align*}
   Similar to Dirichlet's hyperbola method, we split $S_N(x)$ into three sums by first suming over one of the $k,l$ being  $\leq \sqrt{1+N}x$
and then subtracting the overcount arising from $\max{k,l} \leq \sqrt{1+N}x$. Thus, we obtain
\begin{equation} \label{Sx}
  S_N(x) =2R_N(x)-Q_N(x)-T_N(x)
\end{equation}
where
\[  R_N(x) =\sum_{k \leq \sqrt{1+N}x}\sum_{\substack{m \leq x, n \leq x \\ n^2+Nm^2 \equiv 0 \bmod k}} 1,  \quad  Q_N(x) = \sum_{k \leq Nx/\sqrt{1+N}}\sum_{\substack{n^2+Nm^2 \leq kx\sqrt{1+N}\\ n^2+Nm^2 \equiv 0 \bmod k }}1 \]
and
\[ T_N(x) = \sum_{Nx/\sqrt{1+N} < k \leq \sqrt{1+N}x}\sum_{\substack{m \leq x, n \leq x\\ n^2+Nm^2 \leq kx\sqrt{1+N}\\ n^2+Nm^2 \equiv 0 \bmod k }}1 . \]
Here we have further written the subtracting part as the sum of $Q_N(x)$ and $T_N(x)$ so that the conditions $m \leq x$, and $n \leq x$
 are naturally satisfied for $Q_N(x)$.

\subsection{Treatment of $R_N(x)$}
\label{secR}

   We write $(m^2, k) = ab^2$, where $a$ is square-free so that $n^2 + Nm^2 \equiv 0  \pmod k$ implies that
$ab | (m, n)$. Thus replacing $m, n, k$ by $abm, abn, ab^2d$ respectively and noting that the condition $(m^2,k)=ab^2$ becomes $(am, d)=1$, we can recast $R_N(x)$ as
\begin{align*}
  R_N(x) =\sum_{\substack{ab^2d \leq \sqrt{1+N}x \\ (a,d)=1}}\mu^2(a)\sum_{\substack{m \leq x/ab \\ (m,d)=1}} \
\sum_{\substack{n \leq x/ab \\ n^2+Nm^2 \equiv 0 \bmod d}}1,
\end{align*}
where $\mu$ denotes the M\'obius function. \newline

  We then deduce that
\begin{equation} \label{RN}
\begin{split}
  R_N(x) =\sum_{\substack{ab^2d \leq \sqrt{1+N}x \\ (a,d)=1}} &\mu^2(a)\sum_{\substack{m \leq x/ab\\ (m,d)=1}}
\sum_{\substack{v \bmod d \\ v^2+Nm^2 \equiv 0 \bmod d}}
\sum_{\substack{n \leq x/ab \\ n \equiv v \bmod d}}1 \\
=& \sum_{\substack{ab^2d \leq \sqrt{1+N}x \\ (a,d)=1}}\mu^2(a)\sum_{\substack{m \leq x/ab\\ (m,d)=1}}
\sum_{\substack{v \bmod d \\ v^2+N\equiv 0 \bmod d}}\left (
\left [\frac {x/ab-mv}{d} \right ]-\left [\frac {-mv}{d}\right ] \right ) = R_{0,N}(x)+E_{R, N}(x) ,
\end{split}
\end{equation}
where
\[ R_{0,N}(x) =  x\sum_{\substack{ab^2d \leq \sqrt{1+N}x \\ (a,d)=1}}\frac {\mu^2(a)\rho_{0, N}(d)}{abd}\sum_{\substack{m \leq x/ab\\ (m,d)=1}}1 \]
and
\[ E_{R,N}(x) =\sum_{\substack{ab^2d \leq \sqrt{1+N}x \\ (a,d)=1}}\mu^2(a)\sum_{\substack{m \leq x/ab\\ (m,d)=1}} \ \sum_{\substack{v \bmod d \\ v^2+N\equiv 0 \bmod d}} \left ( \psi\left (\frac {-mv}{d}\right)- \psi \left (\frac {x/ab-mv}{d}\right) \right ) \]
We note that it follows from a result of J. D. Vaaler \cite[Theorem 18]{Vaaler} that
for every positive integer $H$, there is a trigonometric polynomial $\psi^*_H$ of degree $H$ such that
\begin{align*}
  \left| \psi(t) -\psi^*_H(t) \right| \leq \frac {1}{2H+2}\sum_{|h| \leq H}\left (1-\frac {|h|}{H+1}\right) e(ht),
\end{align*}
  where
\begin{align}
\label{psistar}
   \psi^*_H(t)=\sum_{1 \leq |h| \leq H} g(h) e(ht),
\end{align}
  with the complex coefficients $g(h)$ satisfying $|g(h)|<|h|^{-1}$. \newline

We then deduce from this that for any integer $H = H(a, b, x) \geq 1$, we have
\begin{equation*}
\begin{split}
E_{R,N}(x) = \sum_{\substack{d \leq \sqrt{1+N}x/ab^2 \\ (a,d)=1}} & \mu^2(a)\sum_{\substack{m \leq x/ab\\ (m,d)=1}}\sum_{\substack{v \bmod d \\ v^2+N\equiv 0 \bmod d}}
\left ( \psi^*_H \left(\frac {-mv}{d} \right )- \psi^*_H \left (\frac {x/ab-mv}{d} \right ) \right )\\
&+O\left (\left |\sum_{\substack{d \leq \sqrt{1+N}x/ab^2 \\ (a,d)=1 }}\ \sum_{\substack{m \leq x/ab\\ (m,d)=1}}\sum_{\substack{v \bmod d \\ v^2+N\equiv 0 \bmod d}}
\sum_{1 \leq |h| \leq H}g(h,d)e \left( \frac {-mhv}{d} \right)\right|+\frac {x^{2+\varepsilon}}{Ha^2b^3}\right ),
\end{split}
\end{equation*}
   where
\[ g(h,d)=\frac 1H\left ( 1-\frac {|h|}{H+1}\right )(1+e \left( \frac {hx}{abd} \right) ). \]

  Note that the last term in the error term comes from the estimation for $g(0,d) \ll 1/H$
   and summing over $h,v,m,d$ trivially, by noting that $\rho_{0,N}(d) \ll 2^{\omega(d)}$, which is bounded above by the number of divisors of $d$. \newline

 It follows from \eqref{psistar} that in order to estimate $E_{R,N}(x)$,  we need to deal with sums of the form
\begin{align*}
  \sum_{\substack{d \leq \sqrt{1+N}x/ab^2 \\(a,d)=1 }}\sum_{\substack{m \leq x/ab\\ (m,d)=1}}\sum_{\substack{v \bmod d \\ v^2+N\equiv 0 \bmod d}}
\sum_{1 \leq |h| \leq H}g'(h,d)e(\frac {-mhv}{d})
\end{align*}
   where $g'(h,d)=g(h,d)$ or $g(h)(1-e(hx/abd))$. We certainly have $|g'(h,d) |\leq 2|h|^{-1}$.  Now such a sum is, using the M\"obius function to detect the condition $(m,d)=1$,
\begin{equation} \label{Eest}
  \sum_{\substack{rl \leq \sqrt{1+N}x/ab^2 \\ (a,rl)=1 }}\mu(l) \sum_{\substack{n \leq x/abl}} \sum_{\substack{v \bmod {rl} \\ v^2+N\equiv 0 \bmod {rl}}} \sum_{1 \leq |h| \leq H}g'(h,rl)e \left(\frac {-nhv}{r} \right) .
  \end{equation}

   Note that the solutions to the congruence equation
\begin{align*}
     v^2+N\equiv 0 \pmod {rl}
\end{align*}
    are all coming from the solutions to the congruence equation
\begin{align*}
     v^2_1+N\equiv 0 \pmod {r}
\end{align*}
    so that $v$ can be written as
\begin{align*}
     v \equiv v_1+nr \pmod {rl}.
\end{align*}
   Note further that we have
\begin{align*}
  e\left( \frac {-nhv}{r} \right)=e \left( \frac {-nhv_1}{r} \right).
\end{align*}

   As each $v_1$ induces $\rho_{0,N}(rl)/\rho_{0,N}(r)$ solutions, we deduce from Lemma \ref{SDHN} that the expression in \eqref{Eest} is
\begin{align*}
\sum_{\substack{rl \leq \sqrt{1+N}x/ab^2 \\ (a,rl)=1 }} &\mu(l)\frac {\rho_{0,N}(rl)}{\rho_{0,N}(r)} \sum_{\substack{n \leq x/abl}} \sum_{\substack{v \pmod {r} \\ v^2+N\equiv 0 \pmod {r}}}
\sum_{1 \leq |h| \leq H}g'(h,rl)e \left( \frac {-nhv}{r} \right) \\
\ll &  x^{\varepsilon} \sum_{\substack{l \leq x/ab^2 }}\sum_{\substack{r \leq x/ab^2l }} \sum_{\substack{v \pmod {r} \\ v^2+N\equiv 0 \pmod {r}}}
\sum_{1 \leq |h| \leq H}\frac 1{|h|} \left |\sum_{\substack{n \leq x/abl}}e \left( \frac {-nhv}{r} \right) \right | \\
\ll &  x^{\varepsilon} \sum_{\substack{l \leq x/ab^2 }} \left( \frac x{abl}+\frac{x}{ab^2l} \right) \sqrt{\frac{x}{ab^2l}} \ll \frac {x^{3/2+\varepsilon}}{a^{3/2}b^2}.
\end{align*}

    Thus, by choosing $H = 1 +[\sqrt{x}/ab^2]$, we have
\begin{align}
\label{E}
   E_{R,N}(x) \ll \sum_{a,b} \frac {x^{3/2+\varepsilon}}{a^{3/2}b^2}+x^{3/2+\varepsilon} \ll x^{3/2+\varepsilon}.
\end{align}

   Applying the formula
\begin{align*}
    \sum_{\substack{m \leq x/ab \\ (m,d)=1}}1=\frac {\varphi(d)x}{abd}+O(d^{\varepsilon}),
\end{align*}
    we deduce that
\begin{align*}
   R_{0,N}(x)=x^2\sum_{\substack{ab^2d \leq \sqrt{1+N}x \\ (a,d)=1}}\frac {\mu^2(a)\rho_{0,N}(d)\varphi(d)}{(abd)^2}+O(x^{1+\varepsilon}).
\end{align*}

 Now consider the equation $v^2 + Nu^2 \equiv 0 \pmod k$ with $0 < u, v \leq k$. By writing $(u^2,k)=ab^2$ with $a$ square-free and
replacing $u, v, k$ by $abu, abv, ab^2d$, we see that $(au, d)=1$ and that we have $0<u, \; v \leq bd$. From this we easily deduce that
\begin{align*}
  \rho_N(k)=\sum_{\substack{ab^2d=k \\ (a,d)=1}}\mu^2(a)b^2\rho_{0,N}(d)\varphi(d),
\end{align*}
   so that
\begin{align*}
   \sum_{k \leq \sqrt{1+N}x}\frac {\rho_N(k)}{k^2}=\sum_{\substack{ab^2d \leq \sqrt{1+N}x \\ (a,d)=1}}\frac {\mu^2(a)\rho_{0,N}(d)\varphi(d)}{(abd)^2}.
\end{align*}
   We then apply Lemma \ref{rhod} to deduce that
\begin{align*}
    R_{0,N}(x)=x^2\left (2A_N\log x+2\int\limits^{\infty}_1\frac {E_N(t)}{t^3}dt+A_N(\log (N+1)+1) \right )+O(x^{4/3+\varepsilon}).
\end{align*}

   Combining this with \eqref{E}, we obtain via \eqref{RN} that
\begin{align}
\label{RNestm}
    R_{N}(x)=x^2\left (2A_N\log x+2\int\limits^{\infty}_1\frac {E_N(t)}{t^3}dt+A_N(\log (N+1)+1) \right )+O \left( x^{3/2+\varepsilon} \right).
\end{align}

\subsection{Treatment of $Q_N(x)$}
   For any integer $k \geq 0$, we denote $r_N(k)$ the number of representations of $k$ as a sum of $n^2+Nm^2$ with $n, m \geq 0$. Note that when
$k$ is not a perfect square or $N$ times a perfect square, then any representation of $k$ of the form $n^2+Nm^2$ must satisfy $nm \neq 0$. From this, we deduce that
when $k$ is not a perfect square or $N$ times a perfect square, then $4r_N(k)$ gives the number of lattice points $(u,v)$ satisfying $u^2+Nv^2=k$. From the well-known
relation (note that in our case the only units in $\mq(\sqrt{-N})$ are $\pm 1$)
\begin{align}
\label{latcount}
   2\sum_{d|n}\chi_N(d)=\sum_{\substack{i^2+Nj^2 = n \\ i,j \in \mz}}1,
\end{align}
   we deduce that when $k$ is not a perfect square or $N$ times a perfect square, then
\begin{align*}
   r_N(n)=\frac 12\sum_{d|n}\chi_N(d).
\end{align*}
   Thus we have
\begin{align*}
  Q_N(x)& =\sum_{k \leq Nx/\sqrt{1+N}}\ \sum_{n \leq \sqrt{1+N}x}r_N(nk) 
 =\frac 12\sum_{k \leq Nx/\sqrt{1+N}} \ \sum_{n \leq \sqrt{1+N}x} \ \sum_{d|nk}\chi_N(d)+O\left( x^{1+\varepsilon} \right) \\
&=\frac 12 \sum_{d \leq Nx^2}\chi_N(d)\sum_{k \leq Nx/\sqrt{1+N}} \ \sum_{\substack{n \leq \sqrt{1+N}x \\ nk \equiv 0 \bmod d}}1+O\left( x^{1+\varepsilon} \right).
\end{align*}
  We write $(n, d) = d_1$ with $d = d_1d_2$, $n = d_1l$ to see that
\begin{align*}
  Q_N(x) &=\frac 12\sum_{d_1 \leq \sqrt{1+N}x}\chi_N(d_1)\sum_{d_2 \leq Nx/\sqrt{1+N}}\chi_N(d_2)
\sum_{\substack{l \leq \sqrt{1+N}x/d_1 \\ (l,d_2)=1}}
\sum_{\substack{k \leq Nx/\sqrt{1+N} \\ d_2|k}}1+O(x^{1+\varepsilon}) \\
&=\frac 12\sum_{d_2 \leq Nx/\sqrt{1+N}}\chi_N(d_2)\left[ \frac {Nx}{d_2\sqrt{1+N}}\right] \sum_{\substack{l \leq \sqrt{1+N}x \\ (l,d_2)=1}} \
\sum_{\substack{d_1 \leq \sqrt{1+N}x/l}}\chi_N(d_1)
+O(x^{1+\varepsilon}).
\end{align*}

  Using the M\"obius function to eliminate the restriction $(l, d_2) = 1$  and writing $d_2 = ms, l = mt$, we obtain
\begin{equation} \label{Q}
\begin{split}
 Q_N& (x)  \\
 =  & \frac 12  \sum_{m \leq Nx/\sqrt{1+N}}\mu(m)\chi_N(m)\sum_{s \leq Nx/m\sqrt{1+N}}\chi_N(s)\left[ \frac {Nx}{ms\sqrt{1+N}} \right]\sum_{\substack{t \leq \sqrt{1+N}x/m }} \
\sum_{\substack{d_1 \leq \sqrt{1+N}x/mt}}\chi_N(d_1)
+O(x^{1+\varepsilon}).
\end{split}
\end{equation}

   We note that, as a special case of counting integral lattice points inside an ellipse (see \cite{Huxley1}), we have for for $x>0$,
\begin{align*}
   \sum_{\substack{i^2+Nj^2 \leq x\\ (i,j) \in \mz}}1 =\frac {\pi x}{\sqrt{N}}+O \left( x^{\alpha} \right),
\end{align*}
   for some constant $\alpha<1/3$. \newline

  We apply this and \eqref{latcount} to see that the inner double sum in \eqref{Q} is
\begin{equation}
\label{InnerQ}
\begin{split}
   \sum_{\substack{t \leq \sqrt{1+N}x/m }} \
\sum_{\substack{d_1 \leq \sqrt{1+N}x/mt}}\chi_N(d_1)&=\sum_{\substack{n \leq \sqrt{1+N}x/m }}
\sum_{\substack{d|n}}\chi_N(d) \\
&=\frac 12\sum_{\substack{i^2+Nj^2 \leq \sqrt{1+N}x/m \\ i,j \in \mz}}1 =\frac {\pi \sqrt{1+N}x}{2m\sqrt{N}}+O\left( \left( \frac {x}{m} \right)^{\alpha+\varepsilon} \right).
\end{split}
\end{equation}

   The same argument shows that
\begin{equation} \label{outerQ}
\begin{split}
  \sum_{s \leq Nx/m\sqrt{1+N}}\chi_N(s)\left[\frac {Nx}{ms\sqrt{1+N}}\right]&=\sum_{s \leq Nx/m\sqrt{1+N}}\chi_N(s)\sum_{r \leq Nx/ms\sqrt{1+N}}1 \\
&=\frac 12\sum_{\substack{i^2+Nj^2 \leq Nx/m\sqrt{1+N} \\ i,j \in \mz}}1
=\frac {\pi \sqrt{N} x}{2\sqrt{1+N}m}+O\left( \left( \frac {x}{m} \right)^{\alpha+\varepsilon} \right).
\end{split}
\end{equation}

   We then conclude from \eqref{Q}, \eqref{InnerQ} and \eqref{outerQ} that
\begin{align}
\label{Qx}
   Q_N(x)=\frac {\pi^2x^2}{8}\sum_{m \leq Nx/\sqrt{1+N}}\frac {\mu(m)\chi_N(m)}{m^2}+O \left( x^{1+\alpha+\varepsilon} \right)
=\frac {\pi^2x^2}{8L(2, \chi_N)}+O \left( x^{1+\alpha+\varepsilon} \right).
\end{align}

\subsection{Treatment of $T_N(x)$}
   First we have
\begin{equation} \label{T}
T_N(x) = T_{N,1}(x)-T_{N,2}(x),
\end{equation}
where
\[ T_{N,1} = \sum_{Nx/\sqrt{1+N} < k \leq \sqrt{1+N}x}\sum_{\substack{m \leq x, n \leq x\\ n^2+Nm^2 \equiv 0 \pmod k }}1 \]
and
\[ T_{N,2} =
\sum_{Nx/\sqrt{1+N} < k \leq \sqrt{1+N}x}\sum_{\substack{\sqrt{(kx\sqrt{1+N}-x^2)/N} < m \leq x}}\sum_{\substack{\sqrt{kx\sqrt{1+N}-Nm^2} < n \leq x}}1 . \]

    The treatment of  $T_{N,1}(x)$ is similar to that of $R_N(x)$, so we omit the details here and we obtain
\begin{align}
\label{T1}
  T_{N,1}(x)=x^2\sum_{\substack{Nx/\sqrt{1+N}< ab^2d \leq \sqrt{1+N}x \\ (a,d)=1}}\frac {\mu^2(a)\rho_{0,N}(d)\varphi(d)}{(abd)^2}+O(x^{3/2+\varepsilon}).
\end{align}

    To deal with $T_{N,2}(x)$, we write $(m^2, k) = ab^2$ with $a$ square-free. Then replacing $m, n, k$ by $abm, abn, ab^2d$ respectively and
noting that the condition $(m^2, k) = ab^2$ becomes
$(am, d) = 1$, we can recast $T_{N,2}(x)$ as
\begin{equation}
\label{T2}
T_{N,2}(x) = \sum_{\substack{Nx/\sqrt{1+N} < ab^2d \leq \sqrt{1+N}x \\ (a,d)=1}}\mu^2(a)\sum_{\substack{A < m \leq x/ab \\ (m,d)=1}} \sum_{\substack{B(m) < n \leq x/ab \\  n^2+Nm^2 \equiv 0 \bmod d}} 1 
\end{equation}
where, for notation convenience, we set
\[  A = \frac{\sqrt{(ab^2dx\sqrt{1+N}-x^2)}}{ab\sqrt{N}} \quad \mbox{and} \quad B(m) = \sqrt{\frac{dx}{a}\sqrt{1+N}-Nm^2} . \]
Now we get
\begin{align}
T_{N,2} (x)=&  \sum_{\substack{Nx/\sqrt{1+N} < ab^2d \leq \sqrt{1+N}x \\ (a,d)=1}}\mu^2(a)\sum_{\substack{A < m \leq x/ab \\ (m,d)=1}}
\sum_{\substack{ v \bmod d \\ v^2+Nm^2 \equiv 0 \bmod d}} \sum_{\substack{B(m) < n \leq x/ab \\ n \equiv v \bmod d}}1 \nonumber \\
=& \sum_{\substack{Nx/\sqrt{1+N} < ab^2d \leq \sqrt{1+N}x \\ (a,d)=1}}\mu^2(a)\sum_{\substack{A < m \leq x/ab \\ (m,d)=1}}
\sum_{\substack{ v \bmod d \\ v^2+Nm^2 \equiv 0 \bmod d}}  \left ( \left[ \frac {x/ab-v}{d} \right]-\left[ \frac {B(m)-v}{d} \right] \right )\nonumber \\
= & \quad T_{N,21}(x)+T_{N,22}(x)-T_{N,23}(x). \nonumber
\end{align}
where
\[ T_{N,21}(x) = \sum_{\substack{Nx/\sqrt{1+N} < ab^2d \leq \sqrt{1+N}x \\ (a,d)=1}}\frac {\mu^2(a)}{d} \sum_{\substack{A < m \leq x/ab \\ (m,d)=1}}
\sum_{\substack{ v \bmod d \\ v^2+Nm^2 \equiv 0 \bmod d}}\left ( \frac {x}{ab}-B(m) \right )  ,
\]
\[T_{N,22}(x)=  \sum_{\substack{Nx/\sqrt{1+N} < ab^2d \leq \sqrt{1+N}x \\ (a,d)=1}}\mu^2(a)
\sum_{\substack{A < m \leq x/ab \\ (m,d)=1}} \sum_{\substack{ v \bmod d \\ v^2+Nm^2 \equiv 0 \bmod d}}\psi \left ( \frac {B(m) -v}{d} \right )  \]
and
\[ T_{N,23}= \sum_{\substack{Nx/\sqrt{1+N} < ab^2d \leq \sqrt{1+N}x \\ (a,d)=1}}\mu^2(a) \sum_{\substack{A < m \leq x/ab \\ (m,d)=1}}
\sum_{\substack{ v \bmod d \\ v^2+Nm^2 \equiv 0 \bmod d}}\psi \left( \frac {x/ab-v}{d} \right) . \]
We shall show that $T_{N, 21}(x)$ gives the main term of $T_{N,2}(x)$ and the others are remainders. We first estimate $T_{N,22}(x)$. Note the main contribution for $T_{N,22}(x)$ comes from the terms when $ab^2$ is small: we suppose $ab^2 \leq x^{1/2-\varepsilon}$ and the discarded terms contribute $O(x^{3/2+\varepsilon})$. This is because the sum over $v$ is $\ll 2^{\omega(d)}$ and the sum over $m$ is $\ll xb/ab^2 \ll x^{1/2+\varepsilon}b$, which implies that if we sum over $d$, then the total sum is $\ll x^{3/2+\varepsilon}/ab$. Now summing trivially over $a,b \ll x$ gives the desired error bound.  Note that we may also shorten the range for $d$ to $N(x +\sqrt{x})/\sqrt{1+N} < ab^2d \leq \sqrt{1+N}x$ as similar discussions as above imply that the range of $Nx/\sqrt{1+N} < ab^2d \leq N(x +\sqrt{x})/\sqrt{1+N}$ gives us an error $O(x^{3/2+\varepsilon})$.  We now write

\begin{align}
\label{T22decomp}
    T_{N,22}(x)=\sum_{a,b}T_{N,22}(a,b,x).
\end{align}

   Similar to the detailed discussion concerning the Fourier approximation of $E_{N,R}(x)$ given in
Section \ref{secR}, for some parameter $H_1 = H_1(a, b)$ which satisfies $x^{\varepsilon} \ll H_1 \ll x^{1/2-\varepsilon}$ and will be
fixed later,
\begin{equation}
\label{T22}
\begin{split}
   T_{N,22}&(a,b,x) \\
    \ll & \left| \sum_d\sum_m\sum_v\sum_{1 \leq |h_1| \leq H_1}\tilde{g}(h_1)e\left(\frac {h_1\left (B(m)-v \right )}{d}\right ) \right|
+\sum_d\sum_m\sum_v\frac {1}{H_1} = \left|  T'_{N,22}(a,b,H_1, x) \right|+O \left(\frac {x^{2+\varepsilon}}{a^2b^3H_1} \right),
\end{split}
\end{equation}
where $\tilde{g}(h_1)$ is certain complex number satisfying $|\tilde{g}(h_1)| \ll 1/|h_1|$.  Using $\sum_{u \bmod {d}^*}$ to denote sum over reduced residue classes modulo $d$, we get
\begin{align*}
   T'_{N,22}(a,b,H_1, x) =& \sum_d\sum_{1 \leq |h_1| \leq H_1}\tilde{g}(h_1)\sum_{\substack{u \bmod {d}^*  \\ v \bmod {d} \\ d | v^2+Nu^2}}e \left( \frac {-h_1v}{d} \right) \sum_{m \equiv u \bmod d}e\left (\frac {h_1B(m)}{d}\right )  \\
=& \sum_d\frac 1{d} \sum_{-d/2<h_2\leq d/2} \ \sum_{1 \leq  |h_1| \leq H_1}\tilde{g}(h_1)\sum_{\substack{u \bmod {d}^*  \\ v \bmod {d}\\ d | v^2+Nu^2}}e \left( \frac {h_2u-h_1v}{d} \right) \sum_{m } e\left (\frac {h_1B(m)-h_2m}{d}\right ) \\
=& \sum_d\frac 1{d} \sum_{-d/2<h_2\leq d/2} \ \sum_{1 \leq  |h_1| \leq H_1}\tilde{g}(h_1)\sum_{\substack{v \bmod {d} \\ v^2+N \equiv 0 \bmod d}}R(h_2-h_1v;d)
\sum_{m }e \left( \frac {h_1B(m)-h_2m}{d} \right),
\end{align*}
  where $R(w; d)$ is the Ramanujan sum
\begin{align}
\label{Rsum}
   R(w; d) =: \sum_{\substack{0<u<d\\ (u,d)=1}}e \left( \frac {wu}{d} \right)=\sum_{s|(w,d)} s\mu \left( \frac ds \right).
\end{align}
  We use \eqref{Rsum} for $R(h_2 -h_1v;d)$ to estimate the
contribution from those $h_2$'s satisfying $|h_2| \leq H_1x^{\varepsilon}$ and it is
\[ \ll \frac x{ab}\sum_d\frac 1{d} \sum_{|h_2| \leq H_1x^{\varepsilon}}\sum_{1 \leq  |h_1| \leq H_1}\frac {1}{|h_1|}
\sum_{\substack{v \bmod {d} \\ v^2+N \equiv 0 \bmod d}}\sum_{s|(d, h_2-h_1v)}s. \]

  Since $v^2 + N \equiv 0 \pmod d$, the condition $s|(h_2 -h_1v)$ implies that $s | (Nh_1^2 + h_2^2)$, thus the above is
\begin{align*}
   \ll & \frac {x^{1+\varepsilon}}{ab}\sum_d\frac 1{d} \sum_{|h_2| \leq H_1x^{\varepsilon}}\sum_{1 \leq  |h_1| \leq H_1}\frac {1}{|h_1|}
\sum_{s|(d, Nh^2_1+h^2_2)}s \ll  \frac {x^{1+\varepsilon}}{ab}\sum_{1 \leq  |h_1| \leq H_1}\frac {1}{|h_1|} \sum_{|h_2| \leq H_1x^{\varepsilon}}
\sum_{\substack{s| Nh^2_1+h^2_2 \\ s \leq \frac {\sqrt{1+N}x}{ab^2}}}\sum_{l \leq \frac {\sqrt{1+N}x}{ab^2s}}\frac 1{l}  \ll \frac {x^{1+\varepsilon}H_1}{ab}.
\end{align*}

   Hence we have
\begin{equation} \label{T'}
\begin{split}
   &T'_{N,22}(a,b, H_1, x) \\
   & \ll x^{\varepsilon}\sum_d\frac 1{d} \max_{\substack{H'_1 \ll H_1 \\ H_1x^{\varepsilon} \ll H_2 \ll d/4}} \frac 1{H'_1}
\sum_{h_1 \sim H'_1}\sum_{H_2<|h_2| \leq 2H_2}\sum_{\substack{v \bmod {d} \\ v^2+N \equiv 0 \bmod d}} \left| R(h_2-h_1v;d) \sum_{m }e \left( \frac {h_1B(m)-h_2m}{d} \right)\right |+\frac {x^{1+\varepsilon}H_1}{ab}. 
\end{split}
\end{equation}
For fixed $H'_1$, $H_2$ and $d$, we can divide the range of $m$ into two parts, say $\Omega_1$ and $\Omega_2$,
such that
\begin{align} \label{omega1}
   \left\| \frac {h_1m/B(m)+h_2}{d} \right\|< \frac {H_1}{d}, \quad \text{if} \; x \in \Omega_1.
\end{align}
   and otherwise if $m \in \Omega_2$. It is clear that $\Omega_1$  consists of at most
$O\left( 1 + H'_1x^{1/4}/d \right)$ continuous segments, by noting that when $d > N(x +\sqrt{x})/(\sqrt{1+N}ab^2)$ we
have $\sqrt{d\sqrt{1+N}/a-Nm^2}> x^{3/4}/ab$.
Now the Kusmin-Landau inequality (see \cite[Lemma 4]{Yu}) asserts that for a differentiable function $f$ with $\| f' \| \geq \lambda>0$ on an interval $I$, then
\begin{align*}
   \sum_{n \in I}e(f(n)) \ll \lambda^{-1}.
\end{align*}
   We then deduce from \eqref{omega1} that for the subsum over each segment contained in $\Omega_2$,  we have a contribution of $O(d/H_1)$, and thus the summation over $\Omega_2$ is bounded by
$O(x^{1/4} + d/H_1)$. \newline

    Since $H_1x^{\varepsilon} \ll 2|h_2| \leq d$ and $h_1$ is positive, we have
\begin{align*}
   \left| \Omega_1 \right| \ll \bigcup_{0 \leq j \ll H'_1x^{1/4}} \left| \Omega_{1j} \right|
\end{align*}
   where $\Omega_{1j}$ are pairwise disjoint segments such that $m \in \Omega_{1j}$ if and only if
\begin{align*}
   \left| \frac {h_1m/B(m)+h_2}{d}-j \right| < \frac {H_1}{d}, \quad \text{if $x \in \Omega_1$}.
\end{align*}

    It can be seen that the above is equivalent to
\begin{align*}
    \frac {dx\sqrt{1+N}}{aN} \left( 1+\frac {h^2_1}{N(jd-h_2-H_1)^2} \right)^{-1}< m^2 < \frac {dx\sqrt{1+N}}{aN} \left( 1+\frac {h^2_1}{N(jd-h_2+H_1)^2} \right)^{-1}.
\end{align*}
    As $d \geq 2|h_2| \gg H_1x^{\varepsilon}$, the above implies that the length of $\Omega_{1j}$ is
\begin{align*}
    \ll h^2_1\sqrt{\frac{dx}{a}} \left( \frac {1}{(jd-h_2-H_1)^2}-\frac {1}{(jd-h_2+H_1)^2} \right)    \ll  \sqrt{\frac{dx}{a}}\frac {H_1h^2_1}{|jd-h_2|^3}.
\end{align*}
     We note that when $j \geq 1$, $|jd-h_2| \gg jd$ so that the total length of $\Omega_1$ is
\begin{align*}
    \ll \Omega_{10}+\sum_{j \geq 1}\sqrt{\frac{dx}{a}}\frac {H_1h^2_1}{(jd)^3} \ll \sqrt{\frac{dx}{a}}\frac {H_1{H'}^2_1}{H^3_2} .
\end{align*}
  Combining this with the estimate over $\Omega_{2}$, we get
\begin{align}
\label{summ}
    \left|\sum_{m }e \left( \frac {h_1B(m) -h_2m}{d} \right) \right| \ll \sqrt{\frac{dx}{a}}\frac {H_1{H'}^2_1}{H^3_2}+x^{1/4} + \frac d{H_1}.
\end{align}

    From \eqref{Rsum}, \eqref{T'} and \eqref{summ} we deduce that
\begin{align*}
 & T'_{N,22}(a,b, H_1, x) \\
\ll & x^{\varepsilon}\sum_d\frac 1{d} \max_{\substack{H'_1 \ll H_1 \\ H_1x^{\varepsilon} \ll H_2 \ll d/4}}
\left ( \sqrt{\frac{dx}{a}}\frac {H_1H'_1}{H^3_2}+\frac {x^{1/4}}{H'_1} + \frac d{H_1H'_1}\right )
\sum_{h_1 \sim H'_1}\sum_{H_2<|h_2| \leq 2H_2}\sum_{\substack{v \bmod {d} \\ v^2+N \equiv 0 \bmod d}}\sum_{s|(v, h_2-h_1v)}s+\frac {x^{1+\varepsilon}H_1}{ab} \\
\ll & x^{\varepsilon} \max_{\substack{H'_1 \ll H_1 \\ H_1x^{\varepsilon} \ll H_2 \ll x/ab^2}}\sum_d\frac 1{d}
\left ( \frac {xH_1H'_1}{abH^3_2}+\frac {x^{1/4}}{H'_1} + \frac x{ab^2H_1H'_1}\right )
\sum_{h_1 \sim H'_1}\sum_{H_2<|h_2| \leq 2H_2} \ \sum_{s|(d, Nh^2_1+h^2_2)}s+\frac {x^{1+\varepsilon}H_1}{ab}
\end{align*}
Now after a short computation, we get
\[ T'_{N,22}(a,b, H_1, x) \ll  \frac {x^{1+\varepsilon}H_1}{ab}+\frac {x^{5/4+\varepsilon}}{ab^2}+\frac {x^{2+\varepsilon}}{(ab^2)^2H_1}+\frac {x^{1+\varepsilon}H_1}{ab}. \]

   Let $H_1 = x^{1/2}/b$, then we have from \eqref{T22} and the above estimation that
\begin{align*}
  T_{N,22}(a,b, H_1, x) \ll \frac {x^{3/2+\varepsilon}}{ab^2}.
\end{align*}
It then follows from \eqref{T22decomp} that
\begin{align} \label{T22bound}
  T_{N,22}(x) \ll x^{3/2+\varepsilon}.
\end{align}

For $T_{N,23}(x)$, we need to estimate a sum of form
\begin{align*}
  \sum_{a,b,d}\sum_{\substack{v \bmod {d} \\ v^2+N \equiv 0 \bmod d}}\sum_{h}a(h)e\left ( \frac {hx/ab-hmv}{d} \right ).
\end{align*}

    This is of the same shape as the sum involved in $E_{R,N}(x)$. By using arguments similar to what we have
done for $E_{R,N}(x)$, one can prove that
\begin{align} \label{T23bound}
  T_{N,23}(x) \ll x^{3/2+\varepsilon}.
\end{align}

    We now evaluate $T_{N,21}(x)$, which will give the main term of $T_{N,2}(x)$. We write for brevity,
\begin{align*}
  f_N(a,b,d,m)=\frac {x}{ab}-\sqrt{dx\sqrt{1+N}/a-Nm^2}.
\end{align*}

    Then we have
\begin{equation}
\begin{split}
\label{T21}
T_{N,21}(x) =&\sum_{\substack{Nx/\sqrt{1+N} < ab^2d \leq \sqrt{1+N}x \\ (a,d)=1}}\frac {\mu^2(a)}{d} \sum_{\substack{u \bmod {d}^* \\ v \bmod {d} \\ v^2+Nu^2 \equiv 0 \bmod d}} \sum_{\substack{A < m \leq x/ab \\ m \equiv u \bmod d}} f_N(a,b,d,m) \\
=&\sum_{\substack{Nx/\sqrt{1+N} < ab^2d \leq \sqrt{1+N}x \\ (a,d)=1}} \frac {\mu^2(a)}{d^2} \sum_{\substack{u \bmod {d}^* \\ v \bmod {d} \\ v^2+Nu^2 \equiv 0 \bmod d}} \sum_{-d/2<h \leq d/2} e \left( -\frac {hu}{d} \right)  \sum_{A < m \leq x/ab } f_N(a,b,d,m)e \left( \frac {hm}{d} \right) \\
= &\quad  T^{(1)}_{N,21}(x)+T^{(2)}_{N,21}(x),
\end{split}
\end{equation}
where
\[ T^{(1)}_{N,21}(x) = \sum_{\substack{Nx/\sqrt{1+N} < ab^2d \leq \sqrt{1+N}x \\ (a,d)=1}}\frac {\mu^2(a)\rho_{0,N}(d)\varphi(d)}{d^2}
\sum_{\substack{A < m \leq x/ab }}f_N(a,b,d,m) \]
and
\[ T^{(2)}_{N,21}(x) = \sum_{\substack{Nx/\sqrt{1+N} < ab^2d \leq \sqrt{1+N}x \\ (a,d)=1}}\frac {\mu^2(a)\rho_{0,N}(d)}{d^2}
\sum_{h \neq 0} R(-h;d) \sum_{A < m \leq x/ab } f_N(a,b,d,m)e \left( \frac {hm}{d} \right) . \]

Note that for $0 < |h| \leq  d/2$,
\begin{align*}
  \sum_{m \leq t}e \left( \frac {hm}{d} \right) \ll \frac {d}{|h|}.
\end{align*}
   We then deduce via partial summation that
\begin{align*}
  \sum_{A < m \leq x/ab }f_N(a,b,d,m)e \left( \frac {hm}{d} \right) \ll \frac {dx}{ab|h|}.
\end{align*}

   Hence with \eqref{Rsum} we have
\begin{equation} \label{T212}
\begin{split}
  T^{(2)}_{N,21}(x) \ll & x^{1+\varepsilon}\sum_{\substack{Nx/\sqrt{1+N} < ab^2d \leq \sqrt{1+N}x \\ (a,d)=1}}\frac {1}{abd}
\sum_{1 \leq h \leq d/2}\frac 1h \sum_{s|(h,d)}s \\
\ll & x^{1+\varepsilon}\sum_{\substack{N x/\sqrt{1+N} < ab^2d \leq \sqrt{1+N}x \\ (a,d)=1}}\frac {1}{ab}\sum_{s \leq \sqrt{1+N}x/ab^2}\frac 1s
\sum_{l \sim x/ab^2s}\frac 1l
\sum_{h' \leq l/2}\frac 1{h'} \ll  x^{1+\varepsilon}.
\end{split}
\end{equation}

   Therefore, from \eqref{T2}, \eqref{T22bound}, \eqref{T23bound}, \eqref{T21} and \eqref{T212}, we have, recalling the definition of $\rho_{0,N}$ in \eqref{rho0def},
\begin{align*}
  T_{N,2}(x)
&=\sum_{\substack{Nx/\sqrt{1+N} < ab^2d \leq \sqrt{1+N}x \\ (a,d)=1}}\frac {\mu^2(a)\rho_{0,N}(d)\varphi(d)}{d^2}\sum_{A < m \leq x/ab } f_N(a,b,d,m)
+ O \left( x^{3/2+\varepsilon} \right)  \\
&=\sum_{\substack{Nx/\sqrt{1+N} < ab^2d \leq \sqrt{1+N}x \\ (a,d)=1}}\frac {\mu^2(a)\rho_{0,N}(d)\varphi(d)}{d^2}
\# \left\{ m \leq \frac x{ab}, \ n \leq \frac x{ab}: n^2+Nm^2 > \frac {dx\sqrt{1+N}}{a} \right\} +O \left( x^{3/2+\varepsilon} \right).
\end{align*}

    Combining this with \eqref{T} and \eqref{T1}, we derive that
\begin{equation}
\label{Tx0}
\begin{split}
  T_N(x) = \sum_{\substack{Nx/\sqrt{1+N} < ab^2d \leq \sqrt{1+N}x \\ (a,d)=1}}\frac {\mu^2(a)\rho_{0,N}(d)\varphi(d)}{d^2}
\#  & \left\{ m \leq \frac x{ab}, n \leq \frac x{ab}: n^2+Nm^2 \leq  \frac {dx\sqrt{1+N}}{a} \right\}
\\ &+O \left(x^{3/2+\varepsilon} \right).
\end{split}
\end{equation}

  Let $C$ be a simple convex closed curve, it is well-known (see \cite[Lemma 2.1.1]{Huxley2}) that
\begin{align}
\label{lat}
  \# \{\text{Integral lattice points inside $C$} \}=
\text{Area enclosed by $C$}+O(\text{Length of the boundary of $C$}) .
\end{align}
Consequently,
\begin{equation*}
\begin{split}
  \# \left\{ m \leq \frac x{ab}, n \leq \frac x{ab}: n^2+Nm^2 \leq  \frac {dx\sqrt{1+N}}{a} \right\} - \text{Area of}  & \left\{ (u, v) : 0 \leq u \leq \frac x{ab}, 0 \leq v \leq \frac x{ab},  u^2+Nv^2 \leq \frac {dx\sqrt{1+N}}{a} \right\} \\
&  \ll \sqrt{\frac {dx}{a}}\ll \frac {x}{ab},
\end{split}
 \end{equation*}
as $d \ll x/ab^2$. \newline

   Note that we also have
\begin{align*}
 \text{Area of} &  \left\{ (u, v) : 0 \leq u \leq \frac x{ab}, 0 \leq v \leq \frac x{ab},  u^2+Nv^2 \leq \frac {dx\sqrt{1+N}}{a} \right\}  \\
=& \frac 1{(ab)^2}\text{Area of}  \left\{ (u, v) : 0 \leq u \leq x, 0 \leq v \leq x,  u^2+Nv^2 \leq ab^2dx\sqrt{1+N} \right\}  \\
=& \frac 1{(ab)^2}\left ( \# \left\{ m \leq  x, n \leq x: n^2+Nm^2 \leq  ab^2dx\sqrt{1+N} \right\}+O(x) \right ),
\end{align*}
again noting that $d \ll x/ab^2$. \newline

   The above estimates allow us to derive from \eqref{Tx0} that
\begin{equation} \label{Tx}
\begin{split}
  T_N(x)  =& \sum_{\substack{Nx/\sqrt{1+N} < ab^2d \leq \sqrt{1+N}x \\ (a,d)=1}}\frac {\mu^2(a)\rho_{0,N}(d)\varphi(d)}{(abd)^2}
\# \{ m \leq x, n \leq x: n^2+Nm^2 \leq  ab^2dx\sqrt{1+N} \}  \\
& \hspace*{3cm} +O\left (x\sum_{\substack{Nx/\sqrt{1+N} < ab^2d \leq \sqrt{1+N}x \\ (a,d)=1}}\frac {\mu^2(a)\rho_{0,N}(d)\varphi(d)}{abd^2} + x^{3/2+\varepsilon} \right) \\
=& \sum_{\substack{Nx/\sqrt{1+N} < k \leq \sqrt{1+N}x }}\frac {\rho_N(k)}{k^2}
\# \left\{ m \leq x, n \leq x: n^2+Nm^2 \leq  kx\sqrt{1+N} \right\}+O(x^{3/2+\varepsilon}).
\end{split}
\end{equation}

   By counting the number of lattice points using \eqref{lat}, we get
\begin{equation} \label{lattice}
\begin{split}
\# \{ m \leq x, n \leq x: n^2+Nm^2 \leq  kx\sqrt{1+N} \}  
= \frac {kx\sqrt{1+N}}{2\sqrt{N}} &\left (\arccos \left(\sqrt{1-\frac {x}{k\sqrt{1+N}}}\right)-
\arccos \left(\sqrt{\frac {Nx}{k\sqrt{1+N}}}\right)\right )  \\
&+\frac 1{2} x \sqrt{\frac {kx\sqrt{1+N}-x^2}{N}}+\frac 1{2} x \sqrt{kx\sqrt{1+N}-Nx^2}+O(x). 
\end{split}
\end{equation}

Thus, partial summation gives that
\begin{equation}
\label{ellipticarea}
\begin{split}
&\sum_{\substack{Nx/\sqrt{1+N} < k \leq \sqrt{1+N}x }} \frac {\rho_N(k)}{k^2}\cdot \frac {kx\sqrt{1+N}}{2\sqrt{N}}\left (\arccos \left(\sqrt{1-\frac {x}{k\sqrt{1+N}}}\right)-
\arccos \left(\sqrt{\frac {Nx}{k\sqrt{1+N}}}\right)\right ) \\
= & \frac {x\sqrt{1+N}}{2\sqrt{N}}\int\limits^{\sqrt{1+N}x}_{Nx/\sqrt{1+N}}\left (\arccos \left(\sqrt{1-\frac {x}{t\sqrt{1+N}}}\right)-
\arccos \left(\sqrt{\frac {Nx}{t\sqrt{1+N}}}\right)\right ) \dif \left ( \sum_{\substack{k \leq t }}\frac {\rho_N(k)}{k}\right )  \\
= & -\frac {x\sqrt{1+N}}{2\sqrt{N}}\left ( \sum_{\substack{k \leq Nx/\sqrt{1+N} }}\frac {\rho_N(k)}{k}\right )\arccos \left(\sqrt{1-\frac {1}{N}}\right) \\
& \hspace*{1cm} -\frac {x\sqrt{1+N}}{2\sqrt{N}}\int\limits^{\sqrt{1+N}x}_{Nx/\sqrt{1+N}}\left ( \sum_{\substack{k \leq t }}\frac {\rho_N(k)}{k}\right ) \dif \left (\arccos \left(\sqrt{1-\frac {x}{t\sqrt{1+N}}}\right)- \arccos \left(\sqrt{\frac {Nx}{t\sqrt{1+N}}}\right)\right ) 
\end{split}
\end{equation}
One easily checks that the last integral above equals to
\[ \int\limits^{\sqrt{1+N}x}_{Nx/\sqrt{1+N}}\left ( \sum_{\substack{k \leq t }}\frac {\rho_N(k)}{k}\right )
 \cdot \left (  -\frac {1}{2t\sqrt{\frac {t\sqrt{1+N}}{x}-1}} - \frac {1}{2t\sqrt{ \frac {t\sqrt{1+N}}{Nx}-1}}  \right ) \dif t . \]
Now, using Lemma \ref{rhod}, we see that the expression in \eqref{ellipticarea} is
\begin{equation} \label{ellipticarea2}
 -A_N \sqrt{N} x^2\arccos \left(\sqrt{1-\frac {1}{N}}\right) +A_Nx^2 \left( 2-\sqrt{1-\frac 1N} \right) +O \left( x^{4/3+\varepsilon} \right).
 \end{equation}
 
A similar computation yields
\begin{equation} \label{rectanglearea}
\begin{split}
 \frac x{2} &\sum_{\substack{Nx/\sqrt{1+N} < k \leq \sqrt{1+N}x }}\frac {\rho_N(k)}{k^2}\cdot \left (  \sqrt{\frac {kx\sqrt{1+N}-x^2}{N}}+\sqrt{kx\sqrt{1+N}-Nx^2} \right ) \\
=& 2A_Nx^2 \left( 2-\sqrt{1-\frac {1}{N}}-N^{1/2}\arctan N^{-1/2}-N^{-1/2}\arctan N^{1/2}+N^{-1/2}\arctan (N-1)^{1/2} \right)+O\left( x^{4/3+\varepsilon} \right).
\end{split}
\end{equation}

Now from \eqref{Tx}-\eqref{rectanglearea}, we conclude that
\begin{equation} \label{Txfinal}
\begin{split}
 T&_N(x) \\
=& \frac {A_Nx^2}{\sqrt{N}}\left (6\sqrt{N}-3\sqrt{N-1}- 2N\arccos \left(\sqrt{1-\frac {1}{N}}\right)-2N\arctan N^{-1/2}-2\arctan N^{1/2}+
2\arctan (N-1)^{1/2} \right ) \\
& \hspace*{2in} +O(x^{4/3+\varepsilon}). 
\end{split}
\end{equation}

\subsection{Completion of Proof of Theorem \ref{mainthm}}

   From \eqref{Sx}, \eqref{RNestm}, \eqref{Qx} and \eqref{Txfinal}, we conclude that
\[ S_N(x)=A_1(N)x^2\log x+A_2(N)x+O \left( x^{3/2+\varepsilon} \right), \]
   where
\begin{equation} \label{A1N}
  A_1(N)= 4A_N,
  \end{equation}
 and
 \begin{equation} \label{A2N}
 \begin{split}
A_2(N) =  4 \int\limits^{\infty}_1\frac {E_N(t)}{t^3} \dif t &-\frac {\pi^2}{8L(2, \chi_N)} +
\frac {A_N }{\sqrt{N}} \left (2\sqrt{N}(\log (N+1)+1)+6\sqrt{N}-3\sqrt{N-1} \right )  \\
& - \frac {2A_N }{\sqrt{N}} \left (N\arccos \left(\sqrt{1-\frac {1}{N}}\right)+N\arctan N^{-1/2}+\arctan N^{1/2}-
\arctan (N-1)^{1/2}  \right ).
\end{split}
\end{equation}
   Here $E_N, A_N$ are given in \eqref{EN} and \eqref{A}, respectively. The assertion of Theorem \ref{mainthm} now follows.

\vspace{0.1in}

\noindent{\bf Acknowledgments.} P. G. is supported in part by NSFC grant 11871082 and L. Z. by the FRG grant PS43707 and the Faculty Silverstar Award PS49334.
Parts of this work were done when P. G. visited the University of New South Wales (UNSW) in August 2018. He wishes to thank UNSW for the invitation, financial support and warm hospitality during his pleasant stay.

\bibliography{biblio}
\bibliographystyle{amsxport}

\vspace*{.5cm}

\noindent\begin{tabular}{p{8cm}p{8cm}}
School of Mathematics and Systems Science & School of Mathematics and Statistics \\
Beihang University & University of New South Wales \\
Beijing 100191 China & Sydney NSW 2052 Australia \\
Email: {\tt penggao@buaa.edu.cn} & Email: {\tt l.zhao@unsw.edu.au} \\
\end{tabular}

\end{document}